\documentclass[10pt]{amsart}

\newtheorem{thm}{Theorem}[section]
\newtheorem{prop}[thm]{Proposition}

\newtheorem{lem}[thm]{Lemma}

\theoremstyle{definition}

\numberwithin{equation}{section}

\newtheorem{ex}[thm]{Example}
\newtheorem{defn}[thm]{Definition}

\def\bbC{\mathbb{C}}

\def\bbT{\mathbb{T}}

\def\bbZ{\mathbb{Z}}
\def\bfr{\mathbf{S}}

\def\bfd{\mathbf{d}}

\def\bfr{\mathbf{r}}

\def\bfz{\mathbf{z}}

\def\calL{\mathcal{L}}

\def\ve{\varepsilon}
\def\Ad{\mathrm{Ad}}
\def\sgn{\mathrm{sgn}}

\begin{document}
\bibliographystyle{amsalpha}

\title[Quantum generalized cluster algebras]
{Quantum generalized cluster algebras
and quantum dilogarithms of higher degrees}

\author{Tomoki Nakanishi}
\address{\noindent Graduate School of Mathematics, Nagoya University, 
Chikusa-ku, Nagoya,
464-8604, Japan}
\email{nakanisi@math.nagoya-u.ac.jp}


\date{}
\maketitle
\begin{abstract}
We extend the notion of  the quantization of the coefficients of the
ordinary cluster algebras
to the generalized cluster algebras by Chekhov and Shapiro.
In parallel to the ordinary case, it is tightly integrated with certain generalizations of
the ordinary quantum dilogarithm,
which we call the quantum dilogarithms of higher degrees.
As an application,
we derive the identities of these generalized quantum dilogarithms 
associated with any period of quantum $Y$-seeds. 
\end{abstract}

\section{Introduction}

The {\em generalized cluster algebras\/}
were introduced by Chekhov and Shapiro
\cite{Chekhov11}.
They naturally generalize the (ordinary) cluster algebras by
Fomin and Zelevinsky \cite{Fomin03a}.
The main feature of the generalized cluster algebras 
is the appearance of {\em polynomials\/}
 in the exchange relations of cluster variables and
coefficients, instead of {\em binomials\/} in the ordinary case.
Generalized cluster algebras naturally appear so far in
Poisson dynamics \cite{Gekhtman02}, Teichm\"uller theory
\cite{Chekhov11}, representation theory \cite{Gleitz14},
exact WKB analysis \cite{Iwaki14b}, etc.
It has been shown
in
 \cite{Chekhov11,Nakanishi14a}
 that essentially all important properties
of the ordinary cluster algebras 
are naturally extended to the generalized ones.

In this note we demonstrate that the notion of {\em quantum cluster algebras\/}
is also extended to the generalized ones.
 To be more precise,
there are two kinds of formulations of quantum cluster algebras,
 the one quantizing the {\em cluster variables\/} by \cite{Berenstein05b} and the one quantizing the {\em coefficients\/} by \cite{Fock03,Fock07},
 and  it is known that they are closely related to each other.
Here, we concentrate on the latter one.
As shown by \cite{Fock03,Fock07},
in the ordinary case,
the quantization of the coefficients is tightly integrated with
the {\em quantum dilogarithm\/}  \cite{Faddeev93,Faddeev94}.
Similarly,
in the generalized case,
 it is  tightly integrated with certain generalizations
of the quantum dilogarithm, which we call the {\em quantum dilogarithms of higher degrees}.
As an application,
we derive the identities of these generalized quantum dilogarithms 
associated with any period of quantum $Y$-seeds,
which are also parallel to the ones in the ordinary case.

The main message of this note is
that
the fundamental (and perhaps all) features of the quantum cluster algebras
are also extended to the generalized ones.

\section{Quantum dilogarithms of higher degrees}
To begin with,
let us recall some basic facts about the dilogarithm, the $q$-dilogarithm, and the quantum dilogarithm.
The {\em dilogarithm\/} $\mathrm{Li}_2(x)$ is defined by
\begin{align}
\mathrm{Li}_2(x)=\sum_{n=1}^{\infty} \frac{x^n}{n^2}.
\end{align}
Let $q$ be a formal variable.
The {\em $q$-dilogarithm} is defined
as follows.
\begin{align}
\label{eq:L1}
\calL_{2,q}(x)=
\sum_{n=1}^{\infty}
\frac{x^n}{n(q^n-q^{-n})}
=
\frac{1}{q-q^{-1}}\sum_{n=1}^{\infty}
\frac{x^n}{n[n]_q},
\end{align}
where $[n]_q=(q^n-q^{-n})/(q-q^{-1})$ is the standard $q$-number.
The power series \eqref{eq:L1} converges for $|x|<1$ and $|q|<1$,
and the following asymptotic behavior holds when $q\rightarrow 1^{-}$,
\begin{align}
\label{eq:asym1}
\calL_{2,q}(x) 
\sim
\frac{\mathrm{Li_2}(x)}{q^2-1}\sim
\frac{\mathrm{Li_2}(x)}{\log q^2}.
\end{align}
This is clear form the second expression of $\calL_{2,q}(x)$ in \eqref{eq:L1}
and the property $\lim_{q\rightarrow 1} [n]_q=n$.

Following \cite{Faddeev93,Faddeev94} (up to some convention),
we introduce the {\em quantum dilogarithm\/}  $\Psi_{q}(x)$,
which is a formal power series in $x$
 with coefficients in $\bbC(q)$, as follows.
\begin{align}
\Psi_{q}(x)&=
\prod_{m=0}^{\infty}
\left(
1+q^{2m+1}x
\right)^{-1}.
\end{align}
In particular, the quantum dilogarithm  $\Psi_{q}(x)$
should be distinguished from the $q$-dilogarithm $\calL_{2,q}(x)$.
The formal power series $\Psi_{q}(x)$ is characterized by the
following recursion relation with initial condition,
\begin{align}
\label{eq:rec1}
\Psi_{q}(0)&=1,
\quad
\Psi_{q}(q^{\pm2}x)
=
\left(
1+qx
\right)^{\pm1}
\Psi_{q}(x),
\end{align}
where two relations in the latter equality are equivalent to each other.
A little confusingly, the quantum dilogarithm is actually the exponential of the $q$-dilogarithm;
namely,
\begin{align}
\label{eq:psi1}
\Psi_{q}(x) = \exp\left(
-\calL_{2,q}(-x) 
\right).
\end{align}
 This is easily shown by using the recursion relation \eqref{eq:rec1}.

Alternatively, one may define
the {\em dilogarithm\/} $\mathrm{Li}_2(x)$ by the integral
\begin{align}
\mathrm{Li}_2(x)=-\int_0^x \log(1-y)
\frac{dy}{y}
=
-\int_0^{-x} \log(1+y)
\frac{dy}{y}.
\end{align}
Then we have
\begin{align}
\label{eq:asym2}
\begin{split}
\log \Psi_{q}(x) 
&=-
\sum_{m=0}^{\infty}
\log(1+q^{2m+1}x)\\
&=\frac{-1}{1-q^2}
\sum_{m=0}^{\infty}
\log(1+q^{2m+1}x)
\frac{q^{2m+1}x - q^{2m+3}x}{q^{2m+1}x}\\ 
&\sim
\frac{-1}{1-q^2}
\int_0^{x} \log(1+y)
\frac{dy}{y}
=
\frac{1}{1-q^2}\mathrm{Li}_2(-x)
\quad
 (q\rightarrow 1^-).
\end{split}
\end{align}
This completely agrees with
\eqref{eq:asym1} and \eqref{eq:psi1}.
 
Now let us generalize the quantum dilogarithm $\Psi_{q}(x)$
to the ones with higher degrees.
For any field  $F$, let $F(q)$ be the field of the rational functions in
the variable $q$.

\begin{defn}
\label{defn:dilog1}
Let $F$ be a field,
 let  $d$ be a positive integer,
and let $\bfz=(z_1,\dots,z_{d-1})$ be a $d-1$-tuple of elements
in $F$.
 When $d=1$, $\bfz$ is regarded as the empty sequence $()$.
We set $z_0=z_d=1$.
Then, we define a formal power series
 $\Psi_{d,\bfz,q}(x)$ in $x$  with coefficients in $F(q)$ as follows:
\begin{align}
\label{eq:psi2}
\Psi_{d,\bfz,q}(x)&=
\prod_{m=0}^{\infty}
\left(
\sum_{s=0}^d 
z_s q^{s(2m+1)} x^s
\right)^{-1}.
\end{align}
When $d=1$, it is the usual quantum dilogarithm $\Psi_{1,(), q}(x)=
\Psi_{q}(x)$.
We call $\Psi_{d,\bfz,q}(x)$ the {\em quantum dilogarithm of degree $d$
with coefficients $\bfz$}.
\end{defn}

\begin{prop}
The formal power series  $\Psi_{d,\bfz,q}(x)$ is characterized by the
 the following recursion relation with initial condition:
 \begin{align}
\label{eq:rec2}
\Psi_{d,\bfz,q}(0)&=1,
\\
\label{eq:rec3}
\Psi_{d,\bfz,q}(q^{\pm2}x)
&=
\left(
\sum_{s=0}^{d} z_s q^{\pm s} x^s
\right)^{\pm1}
\Psi_{d,\bfz,q}(x),
\end{align}
where
 two relations in \eqref{eq:rec3} are equivalent to each other.
\end{prop}
\begin{proof}
For example, we have
\begin{align}
\begin{split}
\Psi_{d,\bfz,q}(q^{2}x)
&=
\prod_{m=0}^{\infty}
\left(
\sum_{s=0}^d 
z_s q^{s(2m+1)} q^{2s} x^s
\right)^{-1}\\
&=
\prod_{m=1}^{\infty}
\left(
\sum_{s=0}^d 
z_s q^{s(2m+1)} x^s
\right)^{-1}
=
\left(
\sum_{s=0}^{d} z_s q^{s} x^s
\right)
\Psi_{d,\bfz,q}(x).
\end{split}
\end{align}
The rest of the properties are  easily shown.
\end{proof}

For any integer $a$, let us introduce the {\em sign function\/}
\begin{align}
\sgn(a)=
\begin{cases}
+ & a > 0\\
0 & a =0\\
- & a<0.
\end{cases}
\end{align}
Here and below, we identify the signs $\pm$ with numbers $\pm 1$.

The following formula will be useful later.
\begin{prop} 
\label{prop:rec1}
For any integer $a$, the following equality holds.
\begin{align}
\Psi_{d,\bfz,q}(q^{2a}x)
=
\Biggl(
\prod_{m=1}^{|a|}
\Biggl(
\sum_{s=0}^d
z_s
q^{\sgn(a) (2m-1)s}
x^s
\Biggr)^{\sgn(a)}
\Biggr)
\Psi_{d,\bfz,q}(x).
\end{align}
\end{prop}
\begin{proof}
This is obtained from 
\eqref{eq:rec3} by induction on $a$.
\end{proof}

In some cases
the quantum dilogarithms of higher degrees
are
factorized by the ordinary quantum dilogarithm.

\begin{prop}
\label{prop:fac1}
Factorization formula.
Suppose that the following factorization
\begin{align}
\label{eq:fac1}
\sum_{s=0}^d
z_s  x^s
=
\prod_{s=1}^d (1-w_s x)
\end{align}
occurs for some $w_1,\dots, w_d \in F$.
Then, we have
\begin{align}
\label{eq:fac2}
\Psi_{d,\bfz,q}(x)
=\prod_{s=1}^d \Psi_q(-w_sx).
\end{align}
\end{prop}
\begin{proof}
One can directly observe the factorization in \eqref{eq:fac2}
as 
\begin{align}
\Psi_{d,\bfz,q}(x)=
\prod_{m=0}^{\infty}
\left(
\sum_{s=0}^d 
z_s q^{s(2m+1)} x^s
\right)^{-1}
=
\prod_{m=0}^{\infty}
\prod_{s=1}^d (1-w_s q^{2m+1} x)^{-1}.
\end{align}
Alternatively,
under the assumption \eqref{eq:fac1},
the right hand side of \eqref{eq:fac2} satisfies
\eqref{eq:rec2} and \eqref{eq:rec3}.
Thus, thanks to Proposition
\ref{prop:rec1}, we have \eqref{eq:fac2}.
\end{proof}

\begin{ex} Let us consider the special case
where $F=\bbC$ and 
the coefficients $\bfz$ is trivial, i.e., $\bfz=\mathbf{1}:=(1,\dots,1)$.
In this case we have the factorization
\begin{align}
\label{eq:pro6}
\sum_{s=0}^d x^s
=\prod_{s=1}^d (1-\omega^s x),
\end{align}
where
\begin{align}
\omega = \exp(2\pi i / (d+1)). 
\end{align}
Thus, by Proposition \ref{prop:fac1}, we have
\begin{align}
\label{eq:pro5}
\Psi_{d,\mathbf{1},q}(x)
=
\prod_{s=1}^d
 \Psi_q
(-\omega^s x).
\end{align}
On the other hand,
there is another factorization formula,
\begin{align}
\label{eq:pro3}
\Psi_{d,\mathbf{1},q}(x)=
\Psi_{q^{d+1}}(
-x^{d+1}) 
\Psi_{q}(-x) ^{-1}.
\end{align}
This is due to the following alternative expression of $\Psi_{d,\mathbf{1},q}(x)$,
\begin{align}
\label{eq:psi3}
\Psi_{d,\mathbf{1},q}(x)
=
\prod_{m=0}^{\infty}
\frac{
1- q^{2m+1}x
}{
\displaystyle
1- (q^{2m+1}x)^{d+1}
}.
\end{align}
Therefore, by \eqref{eq:asym1} and \eqref{eq:psi1},
we have the following asymptotic behavior in the limit $q\rightarrow 1^-$:
\begin{align}
\log \Psi_{d,\mathbf{1},q}(x)
\sim
&\frac{1}{1-q^2} \sum_{s=1}^d \mathrm{Li}_2(\omega^s x)\\
\sim&\frac{1}{1- q^2} \left(\frac{1}{d+1} \mathrm{Li}_2( x^{d+1})
-\mathrm{Li}_2( x)
\right).
\end{align}
In fact, these two expressions coincide due to
the well-known identity for $\mathrm{Li}_2(x)$ called the 
{\em factorization formula\/} \cite[Eq.~(1.14)]{Lewin81},
\begin{align}
\label{eq:Lid2}
\frac{1}{d+1}\mathrm{Li}_{2}(x^{d+1})
=
\sum_{s=0}^d
\mathrm{Li}_{2}(\omega^s x).
\end{align}
As a side remark, in view of
\eqref{eq:psi1},
the expressions
\eqref{eq:pro5} and \eqref{eq:pro3} imply
the equality
\begin{align}
\label{eq:Lid1}
\calL_{2,q^{d+1}}(x^{d+1})
=
\sum_{s=0}^d
\calL_{2,q}(\omega^s x),
\end{align}
which is regarded as
the $q$-analogue of \eqref{eq:Lid2}.
The equality \eqref{eq:Lid1} is also  obtained  directly from  
\eqref{eq:L1} and the equality
\begin{align}
\sum_{s=0}^d \omega^{sn}
=
\begin{cases}
d+1 & n \equiv 0 \mod d+1\\
0 & n \not\equiv 0 \mod d+1.\\
\end{cases}
\end{align}
\end{ex}

\begin{ex}
Let us consider the case where $F=\mathbb{C}$
with arbitrary coefficients $\bfz$.
Let us introduce the {\em dilogarithm of degree $d$ with coefficients $\bfz$}
by the integral,
\begin{align}
\mathrm{Li}_{2;d,\bfz}(x)=
-\int_0^{-x} \log\Biggl(
\sum_{s=0}^d
z_s y^s
\Biggr)
\frac{dy}y.
\end{align}
Then,
by the same calculation as in \eqref{eq:asym2},
we have the following asymptotic behavior,
\begin{align}
\label{eq:asym3}
\begin{split}
\log \Psi_{d,\bfz,q}(x) 
&=\frac{-1}{1-q^2}
\sum_{m=0}^{\infty}
\log
\Biggl(
\sum_{s=0}^d
z_s (q^{2m+1}x)^s
\Biggr)
\frac{q^{2m+1}x - q^{2m+3}x}{q^{2m+1}x}\\ 
&\sim
\frac{1}{1-q^2}\mathrm{Li}_{2;d,\bfz}(-x) 
\quad
 (q\rightarrow 1^-).
\end{split}
\end{align}
\end{ex}

\section{Generalized mutations of quantum $Y$-seeds}
In this section, following the idea of
\cite{Fock03,Fock07},
 we introduce the quantum version of the
generalized mutation  of generalized cluster algebras.
Here, we use the formulation of generalized cluster algebras
by \cite{Nakanishi14a}.

Let $B=(b_{ij})_{i,j=1}^n$ be a skew-symmetrizable integer matrix.
Let $\bfd=(d_1,\dots,d_n)$ be an $n$-tuple of
positive integers.
For given $B$ and $\bfd$,
we arbitrarily choose an   $n$-tuple of
positive integers
$\bfr=(r_1,\dots,r_n)$ 
 such that
\begin{align}
\label{eq:skew1}
r_i d_i b_{ij} = - r_jd_j b_{ji}.
\end{align}
Such an $\bfr$ exists (not uniquely) due to the skew-symmetrizable property of the matrix $B$.
Let $q$ continue to be a formal variable, and
let $Y=(Y_i)_{i=1}^n$  be an $n$-tuple of
noncommutative formal variables with
commutation relation
\begin{align}
\label{eq:Ycom1}
Y_i Y_j = q^{2r_j d_j b_{ji}} Y_j Y_i.
\end{align}
The relation \eqref{eq:Ycom1} makes sense
due to the skew-symmetric property in \eqref{eq:skew1}.
We call such a pair $(B,Y)$ a {\em quantum $Y$-seed}.

We use the notation
\begin{align}
q_i:= q^{r_i d_i},
\quad
i=1,\dots,n.
\end{align}
Then, \eqref{eq:Ycom1} is also written as
\begin{align}
\label{eq:Ycom3}
Y_i Y_j = q_j^{2 b_{ji}} Y_j Y_i
=q_i^{-2 b_{ij}} Y_j Y_i.
\end{align}

Let $F$ be any field.
For the above $\bfd=(d_1,\dots,d_n)$ we arbitrarily choose
a collection of elements in $F$,
\begin{align}
\bfz=(z_{i,s})_{i=1,\dots,n; s=1,\dots, d_i-1}
\end{align}
satisfying the {\em reciprocity condition\/} in \cite{Nakanishi14a}
\begin{align}
z_{i,s} = z_{i,d_i-s}.
\end{align}
(The use of symbol $\bfz$ here slightly conflicts with the  one in
Definition \ref{defn:dilog1}, but we find that it is convenient.)
Let us set $z_{i,0}=z_{i,d_i}=1$.
We also introduce the notation
\begin{align}
\bfz_i = (z_{i,s})_{s=1,\dots,d_i-1},
\quad
i=1,\dots,n.
\end{align}
Under these notations we have the associated
 quantum dilogarithm 
$\Psi_{d_i,\bfz_i,q_i}(x)$ of  degree $d_i$
for each $i=1,\dots,n$.

Below we assume that any element in $F$ commutes with variables $Y_i$.

\begin{defn}
For a quantum $Y$-seed $(B,Y)$,
the {\em $(\bfd,\bfz)$-mutation (generalized mutation) $(B',Y')=\mu_{k}(B,Y)$ of
$(B,Y)$ at $k$} is defined by
\begin{align}
\label{eq:bmut1}
b'_{ij}&=
\begin{cases}
-b_{ij} & \text{$i=k$ or $j=k$}\\
b_{ij} + d_k([-b_{ik}]_+  b_{kj}
+b_{ik}[b_{kj}]_+)
& i,j \neq k,
\end{cases}
\\
\label{eq:Ymut1}
Y'_{i}
&=
\begin{cases}
Y_k^{-1} & i = k\\
\displaystyle
q_i^{b_{ik} d_k [\ve b_{ki}]_+}
Y_i Y_k^{d_k[\ve b_{ki}]_+}\\
\displaystyle
\qquad
\times \prod_{m=1}^{|b_{ki}|}
\left(
\sum_{s=0}^{d_k}
z_{k,s}
q_k^{-\ve \sgn(b_{ki})  (2m-1) s}
Y_k^{\ve s}
\right)^{-\sgn(b_{ki})}
& i \neq k,
\end{cases}
\end{align}
where $\ve=\pm$, and
\begin{align}
[a]_+=
\begin{cases}
a & a > 0\\
0 & a \leq 0.
\end{cases}
\end{align}
Actually, the right hand side of \eqref{eq:Ymut1} does not depend on
the choice of the sign $\ve$ (see Lemma \ref{lem:mut1} (i)).
We call $\bfd$  and $\bfz$ the {\em mutation degrees}
and  the {\em frozen coefficients}, respectively, in accordance with
\cite{Nakanishi14a}.
\end{defn}

When we formally set $q=1$, 
the relation \eqref{eq:Ymut1}
reduces to
\begin{align}
\label{eq:ymut1}
Y'_i&=
\begin{cases}
\displaystyle
Y_k^{-1}
 & i=k\\
 \displaystyle
Y_i
Y_k^{d_k[\ve {b}_{ki}]_+} 
\Biggl(
\sum_{s=0}^{d_k}
z_{k,s} Y_k^{\ve s}
\Biggr)^{-{b}_{ki}}
& i\neq k,\\
\end{cases}
\end{align}
which is
 the
generalized mutation of coefficients ($y$-variables)
in generalized cluster algebras formulated in
\cite{Nakanishi14a}.
On the other hand, 
when we set $d_k=1$, it reduces to the  ordinary mutation of 
quantum 
$Y$-seeds by Fock and Goncharov \cite{Fock03,Fock07}.

The following properties are easily checked.
\begin{lem}
\label{lem:mut1}
(i) The right hand side of \eqref{eq:Ymut1} does not depend on
the choice of the sign $\ve$.
\par
(ii) For the matrix $B'$, the condition
\begin{align}
r_i d_i b'_{ij} = - r_jd_j b'_{ji}.
\end{align}
holds.
\par
(iii) The $(\bfd,\bfz)$-mutation is involutive,
i.e., $\mu_k(\mu_k(B,Y))=(B,Y)$.
\end{lem}

In the rest of the section,
we will justify the relation  \eqref{eq:Ymut1}
as a ``good" quantization of 
the classical one \eqref{eq:ymut1}
in the sense of \cite{Fock03,Fock07}.

To start, let us consider
\begin{align}
\Ad(\Psi_{d_k,\bfz_k,q_k}(Y_k^{\ve}))^{\ve}(Y_i)
:=&\ 
\Psi_{d_k,\bfz_k,q_k}(Y_k^{\ve})^{\ve} Y_i
\Psi_{d_k,\bfz_k,q_k}(Y_k^{\ve})^{-\ve},
\end{align}
which we call the {\em adjoint action\/} of
$\Psi_{d_k,\bfz_k,q_k}(Y_k^{\ve})$
on quantum $Y$-variables.

The following is the key formula which connects
the  generalized mutation of quantum $Y$-seeds
and the quantum dilogarithms of higher degrees.

\begin{lem}
\label{lem:ad1}
\begin{align}
\label{eq:ad1}
\begin{split}
\Ad(\Psi_{d_k,\bfz_k,q_k}(Y_k^{\ve}))^{\ve}(Y_i)
=Y_i
 \prod_{m=1}^{|b_{ki}|}
\left(
\sum_{s=0}^{d_k}
z_{k,s}q_k^{-\ve \sgn(b_{ki})  (2m-1) s}
Y_k^{\ve s}
\right)^{-\sgn(b_{ki})}.
\end{split}
\end{align}
\end{lem}
\begin{proof}
For example, in the case $\ve=+$,
\begin{align}
\begin{split}
\Psi_{d_k,\bfz_k,q_k}(Y_k) Y_i
\Psi_{d_k,\bfz_k,q_k}(Y_k)^{-1}
&=
Y_i
\Psi_{d_k,\bfz_k,q_k}(q_k^{-2 b_{ki}}Y_k) 
\Psi_{d_k,\bfz_k,q_k}(Y_k)^{-1}\\
&=
 Y_i
 \prod_{m=1}^{|b_{ki}|}
\left(
\sum_{s=0}^{d_k}
z_{k,s}
q_k^{- \sgn(b_{ki})   (2m-1) s}
Y_k^{ s}
\right)^{-\sgn(b_{ki}) },
\end{split}
\end{align}
where we used 
\eqref{eq:Ycom1} and
Proposition \ref{prop:rec1} in the first and second equalities,
respectively.
The case $\ve=-$ can be shown in the same way.
\end{proof}

The right hand side of \eqref{eq:ad1}, excluding the factor $Y_i$,
 is a part of
 \eqref{eq:Ymut1},
 and for $d_k=1$ it is called the ``automorphism part" of  \eqref{eq:Ymut1}
in \cite{Fock03,Fock07}.

Next let us consider the ``monomial part" of 
 \eqref{eq:Ymut1}.
Let us set
\begin{align}
\label{eq:Ymut2}
Z^{(\ve)}_{i}:=
\begin{cases}
Y_k^{-1} & i = k\\
\displaystyle
q_i^{b_{ik} d_k [\ve b_{ki}]_+}
Y_i Y_k^{d_k[\ve b_{ki}]_+}& i \neq k.
\end{cases}
\end{align}
By Lemma \ref{lem:ad1},
the $(\bfd,\bfz)$-mutation \eqref{eq:Ymut1} is
expressed as the composition
\begin{align}
\label{eq:YZ1}
Y'_i=\Ad(\Psi_{d_k,\bfz_k,q_k}(Y_k^{\ve}))^{\ve}(Z_i^{(\ve)}).
\end{align}

\begin{lem}
\label{lem:Z1}
The following commutation relation holds.
\begin{align}
\label{eq:Zcom2}
Z^{(\ve)}_i Z^{(\ve)}_j = q^{2r_jd_j b'_{ji}} Z^{(\ve)}_j Z^{(\ve)}_i,
\end{align}
where $b'_{ij}$ is given by \eqref{eq:bmut1}.
\end{lem}
\begin{proof}
This is easily verified by the case check.
\end{proof}

\begin{prop}
\label{prop:com1}
Under the $(\bfd,\bfz)$-mutation in \eqref{eq:Ymut1},
the following commutation relation holds:
\begin{align}
\label{eq:Ycom2}
Y'_i Y'_j = q^{2r_jd_j b'_{ji}} Y'_j Y'_i.
\end{align}
\end{prop}
\begin{proof}
By Lemma \ref{lem:Z1} and 
\eqref{eq:YZ1},
we have
\begin{align}
\begin{split}
Y'_i Y'_j &=
\Ad(\Psi_{d_k,\bfz_k,q_k}(Y_k^{\ve}))^{\ve}(Z_i^{(\ve)}Z_j^{(\ve)})\\
&=
\Ad(\Psi_{d_k,\bfz_k,q_k}(Y_k^{\ve}))^{\ve}
(
q^{2r_jd_j b'_{ji}} Z^{(\ve)}_j Z^{(\ve)}_i)\\
&=
q^{2r_jd_j b'_{ji}} Y'_j Y'_i.
\end{split}
\end{align}
\end{proof}
Lemmas \ref{lem:ad1}, \ref{lem:Z1},
and Proposition \ref{prop:com1}
naturally extend
 the  fundamental properties
of the mutation of quantum $Y$-seeds in \cite{Fock03,Fock07}.

\section{Quantum dilogarithm identities of higher degrees}
Let us give an application of generalized mutations of quantum $Y$-seeds
to quantum dilogarithm identities of higher degrees.
Since they are parallel to the one for ordinary quantum dilogarithm identities
studied in \cite{Keller11,Kashaev11},
we only give the minimal description here.
We ask the reader to consult \cite[Section 3]{Kashaev11} for more details.

Consider a sequence of $(\bfd,\bfz)$-mutations  of quantum $Y$-seeds,
\begin{align}
\label{eq:seq1}
(B(1),Y(1))
{\buildrel \mu_{k_1} \over \leftrightarrow }
(B(2),Y(2))
{\buildrel \mu_{k_2} \over \leftrightarrow }
\cdots
{\buildrel \mu_{k_L} \over \leftrightarrow }
(B(L+1),Y(L+1)),
\end{align}
and suppose  that it has the periodicity
\begin{align}
\label{eq:period1}
b_{\sigma(i)\sigma(j)}(L+1)
=
b_{ij}(1),
\quad
Y_{\sigma(i)}(L+1)=
Y_{i}(1)
\end{align}
for some permutation $\sigma$ of $1,\dots,n$.
Then, we have the associated sequence of $(\bfd,\bfz)$-mutations of
(nonquantum) $Y$-seeds
of a (nonquantum) generalized cluster algebra,
\begin{align}
\label{eq:seq2}
(B(1),y(1))
{\buildrel \mu_{k_1} \over \leftrightarrow }
(B(2),y(2))
{\buildrel \mu_{k_2} \over \leftrightarrow }
\cdots
{\buildrel \mu_{k_L} \over \leftrightarrow }
(B(L+1),y(L+1)),
\end{align}
and it has the same periodicity
\begin{align}
\label{eq:period2}
y_{\sigma(i)}(L+1)=
y_{i}(1).
\end{align}
Let us further assume the {\em sign-coherence property\/} of the sequence
\eqref{eq:seq2} (see, .e.g., \cite{Nakanishi14a}).
Let $\ve_t$ and $c_t$ ($t=1,\dots,L$) be the {\em tropical sign\/}
and the {\em $c$-vector\/} of $y_{k_t}(t)$ defined in \cite{Nakanishi14a}.
Let us denote the initial seed $(B(1),Y(1))$ as $(B,Y)$.
Let $\bbT(B)$ be the {\em quantum torus\/}
generated by noncommutative variables $Y^{\alpha}$ ($\alpha\in \bbZ^n$)
with the relations
\begin{align}
q^{\langle \alpha, \beta\rangle} Y^{\alpha} Y^{\beta}
=Y^{\alpha+\beta},
\quad
\langle \alpha, \beta\rangle
= \sum_{i,j=1}^n
\alpha_i r_i d_i b_{ij} \beta_j.
\end{align}
We identify $Y_i = Y^{e_i}$, where $e_i$ is the $i$th unit vector.

\begin{thm}
Under the assumption of the periodicity \eqref{eq:period1}
and the sign-coherence property
of the sequence \eqref{eq:seq2},
we have the following identities of
the quantum dilogarithms of higher degrees
associated to the sequence \eqref{eq:seq1}.
\par
(i) Quantum dilogarithm identities in tropical form
(cf.~\cite[Theorem 3.5]{Kashaev11}).
\begin{align}
\label{eq:id1}
\Psi_{d_{k_1},\bfz_{k_1},q_{k_1}}(Y^{\ve_1c_1})^{\ve_1}
\cdots
\Psi_{d_{k_L},\bfz_{k_L},q_{k_L}}(Y^{\ve_1c_L})^{\ve_L}
=1,
\end{align}
where $Y^{\ve_tc_t}\in \bbT(B)$.
\par
(ii) Quantum dilogarithm identities in universal form
(cf.~\cite[Corollary 3.7]{Kashaev11}).
\begin{align}
\label{eq:id2}
\Psi_{d_{k_L},\bfz_{k_L},q_{k_L}}(Y_{k_L}(L))^{\ve_L}
\cdots
\Psi_{d_{k_1},\bfz_{k_1},q_{k_1}}(Y_{k_1}(1))^{\ve_1}
=1.
\end{align}
\end{thm}

We omit the proof,
since it is completely parallel to
the one for Theorem 3.5 and Corollary 3.7 of \cite{Kashaev11}.

\begin{ex}
Let us consider the simplest nontrivial example of a generalized cluster algebra
with
\begin{align}
B=\begin{pmatrix}
0 & -1\\
1 & 0\\
\end{pmatrix},
\quad
\bfd=(2,1),
\quad
\bfz=(z_{1,1}).
\end{align}
This example was studied in \cite[Section 2.3]{Nakanishi14a}
for the nonquantum case.
Now let us choose 
\begin{align}
\bfr=(1,2).
\end{align}
Thus, we have $q_1=q_2=q^2$, and
the commutation relation for $Y=(Y_1,Y_2)$ is given by
\begin{align}
Y_1Y_2=q^4 Y_2Y_1.
\end{align}
Let us set $(B(1),Y(1)):=(B,Y)$ and consider the following sequence of mutations
\begin{align}
\label{eq:seq3}
\begin{split}
(B(1),Y(1))\
{\buildrel \mu_{1} \over \leftrightarrow }\
(B(2),Y(2))\
&{\buildrel \mu_{2} \over \leftrightarrow }\
(B(3),Y(3))\
{\buildrel \mu_{1} \over \leftrightarrow }
(B(4),Y(4))\\
{\buildrel \mu_{2} \over \leftrightarrow }\
(B(5),Y(5))\
&{\buildrel \mu_{1} \over \leftrightarrow }\
(B(6),Y(6))\
{\buildrel \mu_{2} \over \leftrightarrow }\
(B(7),Y(7)).
\end{split}
\end{align}
Then, we have
\begin{align}
B(t)=(-1)^{t+1} B,
\end{align}
and the quantum $Y$-variables mutate as follows,
where we set $z=z_{1,1}$ for simplicity.
(If we set $q=1$, we recover the result in
\cite[Table 1]{Nakanishi14a} for the nonquantum case.)
\begin{align}
&
\begin{cases}
Y_1(1)=Y_1\\
Y_2(1)=Y_2,\\
\end{cases}
\\
&
\begin{cases}
Y_1(2)=Y_1^{-1}\\
Y_2(2)=Y_2(1+zq^2 Y_1 + q^4 Y_1^2),\\
\end{cases}
\\
&
\begin{cases}
Y_1(3)=Y_1^{-1}(1+q^2Y_2 + zY_1Y_2 + q^{-2}Y_1^2 Y_2)\\
Y_2(3)=Y_2^{-1}(1+zq^{-2} Y_1 + q^{-4} Y_1^2)^{-1},\\
\end{cases}
\\
\allowbreak
&
\begin{cases}
Y_1(4)=Y_1(1+q^{-2}Y_2 + zq^{-4}Y_1Y_2 + q^{-6}Y_1^2 Y_2)^{-1}\\
Y_2(4)=q^{-4}Y_1^{-2}Y_2^{-1}
(1+q^2 Y_2 + q^6 Y_2 + q^8 Y_2^2\\
\qquad\qquad\qquad
+z Y_1Y_2 +zq^2 Y_1 Y_2^2
+q^{-4} Y_1^2 Y_2^2
),\\
\end{cases}
\\
\allowbreak
&
\begin{cases}
Y_1(5)=q^{-2}Y_1^{-1}Y_2^{-1}
(1+q^{2} Y_2)\\
Y_2(5)=q^{-4}Y_1^2 Y_2
(1+q^{-6} Y_2 + q^{-2} Y_2 + q^{-8} Y_2^2\\
\qquad\qquad\qquad
+z q^{-4}Y_1Y_2 +zq^{-10} Y_1 Y_2^2
+q^{-10} Y_1^2 Y_2^2
)^{-1},\\
\end{cases}
\\
\allowbreak
&
\begin{cases}
Y_1(6)=q^{-2}Y_1Y_2
(1+q^{-2} Y_2)\\
Y_2(6)=Y_2^{-1},\\
\end{cases}
\\
\allowbreak
&
\begin{cases}
Y_1(7)=Y_1\\
Y_2(7)=Y_2.\\
\end{cases}
\end{align}
Among them, the calculation of $Y_2(4)$ is the most tedious one.
Now we observe the periodicity of the sequence
\eqref{eq:seq3} with $\sigma=\mathrm{id}$.
We also have the following data of the tropical signs and
the $c$-vectors in \cite[Section 3.4]{Nakanishi14a}
\begin{align}
&\ve_1=\ve_2=+,
\quad
\ve_3=\ve_4=\ve_5=\ve_6=-,\\
\begin{split}
&c_1=(1,0),\
c_2=(0,1),\
c_3=(-1,0),\\
&c_4=(-2,-1),\
c_5=(-1,-1),\
c_6=(0,-1),
\end{split}
\end{align}
which can be also read off from the above result by setting $q=1$.
Now, by substituting these data, the identity \eqref{eq:id1} reads
\begin{align}
\label{eq:id3}
\begin{split}
&\Psi_{2,(z),q^2}(Y_1)
\Psi_{q^2}(Y_2)
\Psi_{2,(z),q^2}(Y_1)^{-1}\\
&\qquad
\times
\Psi_{q^2}(q^{-4}Y_1^2Y_2)^{-1}
\Psi_{2,(z),q^2}(q^{-2}Y_1Y_2)^{-1}
\Psi_{q^2}(Y_2)^{-1}
=1,
\end{split}
\end{align}
while the identity \eqref{eq:id2} reads
\begin{align}
\label{eq:id4}
\begin{split}
&
\Psi_{q^2}(Y_2)^{-1}
\Psi_{2,(z),q^2}((1+q^2Y_2)^{-1}q^2Y_2Y_1)^{-1}
\\
&
\times\Psi_{q^2}(
(1+q^2 Y_2 + q^6 Y_2 + q^8 Y_2^2
+z Y_1Y_2 +zq^2 Y_1 Y_2^2
+q^{-4} Y_1^2 Y_2^2
)^{-1}
q^{4} Y_2Y_1^2)^{-1}
\\
&
\times
\Psi_{2,(z),q^2}((1+q^2Y_2 +zY_1Y_2 + q^{-2}Y_1^2Y_2)^{-1}Y_1)^{-1}
\\
&
\times
\Psi_{q^2}(Y_2(1+zq^2Y_1+q^4Y_1^2))
\Psi_{2,(z),q^2}(Y_1)
=1.
\end{split}
\end{align}
\end{ex}

In general, we conjecture that if the underlying nonquantum sequence
\eqref{eq:seq2} has a periodicity \eqref{eq:period2},
then the the corresponding quantum sequence
\eqref{eq:seq1} also has the same periodicity \eqref{eq:period1}.
(The converse is trivial as already stated.)
This was  proved for the ordinary cluster algebras in
\cite[Proposition 3.4]{Kashaev11} when $B=B(1)$ is
skew-symmetric.
\bibliography{../../biblist/biblist.bib}

\providecommand{\bysame}{\leavevmode\hbox to3em{\hrulefill}\thinspace}
\providecommand{\MR}{\relax\ifhmode\unskip\space\fi MR }
\providecommand{\MRhref}[2]{%
  \href{http://www.ams.org/mathscinet-getitem?mr=#1}{#2}
}
\providecommand{\href}[2]{#2}
\begin{thebibliography}{GSV03}

\bibitem[BZ05]{Berenstein05b}
A.~Berenstein and A.~Zelevinsky, \emph{Quantum cluster algebras}, Adv. in Math.
  \textbf{195} (2005), 405--455; arXiv:math.QA/0404446.

\bibitem[CS14]{Chekhov11}
L.~Chekhov and M.~Shapiro, \emph{Teichm\"uller spaces of {R}iemann surfaces
  with orbifold points of arbitrary order and cluster variables}, Int. Math.
  Res. Notices \textbf{2014} (2014), 2746--2772; arXiv:1111.3963 [math--ph].

\bibitem[FG09a]{Fock03}
V.~V. Fock and A.~B. Goncharov, \emph{Cluster ensembles, quantization and the
  dilogarithm}, Annales Sci. de l'\'Ecole Norm. Sup. \textbf{42} (2009),
  865--930; arXiv:math/0311245 [math.AG].

\bibitem[FG09b]{Fock07}
\bysame, \emph{The quantum dilogarithm and representations of quantum cluster
  varieties}, Invent. Math. \textbf{172} (2009), 223--286; arXiv:math/0702397
  [math.QA].

\bibitem[FK94]{Faddeev94}
L.~D. Faddeev and R.~M. Kashaev, \emph{Quantum dilogarithm}, Mod. Phys. Lett.
  \textbf{A9} (94), 427--434; arXiv:hep--th/9310070.

\bibitem[FV93]{Faddeev93}
L.~D. Faddeev and A.~Yu. Volkov, \emph{Abelian current algebra and the
  {V}irasoro algebra on the lattice}, Phys. Lett. \textbf{315} (1993),
  311--318; arXiv:hep--th/9307048.

\bibitem[FZ03]{Fomin03a}
S.~Fomin and A.~Zelevinsky, \emph{Cluster algebras {II}. {F}inite type
  classification}, Invent. Math. \textbf{154} (2003), 63--121;
  arXiv:math/0208229 [math.RA].

\bibitem[Gle14]{Gleitz14}
A.~Gleitz, \emph{Quantum affine algebras at roots of unity and generalized
  cluster algebras}, 2014, arXiv:1410.2446 [math.RT].

\bibitem[GSV03]{Gekhtman02}
M.~Gekhtman, M.~Shapiro, and A.~Vainshtein, \emph{Cluster algebras and
  {P}oisson geometry}, Moskov Math. J. \textbf{3} (2003), 899--934;
  arXiv:math0208033 [math.QA].

\bibitem[IN14]{Iwaki14b}
K.~Iwaki and T.~Nakanishi, \emph{Exact {WKB} analysis and cluster algebras
  {II}: simple poles, orbifold points, and generalized cluster algebras}, 2014,
  arXiv:1409.4641 [math.CA].

\bibitem[Kel11]{Keller11}
B.~Keller, \emph{On cluster theory and quantum dilogarithm identities},
  Representations of algebras and related topics (A.~Skowro\'nski and
  K.~Yamagata, eds.), EMS Series of Congress Reports, European Mathematical
  Society, 2011, pp.~85--116; arXiv:1102.4148 [math.RT].

\bibitem[KN11]{Kashaev11}
R.~M. Kashaev and T.~Nakanishi, \emph{Classical and quantum dilogarithm
  identities}, SIGMA \textbf{7} (2011), 102, 29 pages; arXiv:1104.4630
  [math.QA].

\bibitem[Lew81]{Lewin81}
L.~Lewin, \emph{Polylogarithms and associated functions}, North-Holland,
  Amsterdam, 1981.

\bibitem[Nak14]{Nakanishi14a}
T.~Nakanishi, \emph{Structure of seeds in generalized cluster algebras}, 2014,
  arXiv:1409.5967 [math.RA].

\end{thebibliography}
\end{document}